\documentclass[10pt,a4paper]{article}
\usepackage{amsmath,amsfonts,amssymb,amsthm,epsfig,epstopdf,titling,url,array}
\usepackage[utf8]{inputenc}
\usepackage{longtable}
\usepackage{setspace}
\usepackage{geometry}                		                	
                                  
\usepackage{graphicx}	
\usepackage{subfigure}
\usepackage{float}	
\usepackage{xcolor}
\usepackage{comment}
\usepackage{url}
\usepackage{hyperref}

\usepackage[english]{babel}

\theoremstyle{definition}
\newtheorem{defi}{Definition}

\theoremstyle{remark}

\newtheorem{rem}[defi]{Remark}

\theoremstyle{plain}
\newtheorem{thm}[defi]{Theorem}
\newtheorem*{B}{Baire Category Theorem}

\newtheorem{lem}[defi]{Lemma}

\title{Adding machines and open dynamical systems}
\author{F. Ciavattini$^*$, T.H. Steele$^{**}$}
\date{\small{$^*$Scuola di Scienze e Tecnologie, Università degli Studi di Camerino, Italy, e-mail: filippo.ciavattini@studenti.unicam.it}\\
$^{**}$ Department of Mathematics, Weber State University,  Ogden UT,  84408, USA, e-mail: thsteele@weber.edu}

\begin{document}
\maketitle
\begin{abstract}
Let $f:\mathcal{M}\rightarrow\mathcal{M}$ be a continuous map defined on a compact metric space $\mathcal{M}$. An open dynamical system introduces disjoint open balls centered at points in $\mathcal{M}$, and considers the trajectories of points from $\mathcal{M}$, and the balls that they visit first. As the balls in question are allowed to shrink, a point is considered indecisive if its trajectory changes infinitely many times the ball first visited. Here, we let $\mathcal{M}$ be an adding machine, a simple system and a solenoidal system. In each case, we show that the set of points which generate indecisive trajectories is residual.\\


\end{abstract}

\section{Introduction}\label{intro}
Let $\mathcal{M}$ be a compact metric space and $f:\mathcal{M}\rightarrow\mathcal{M}$ be a continuous function.
Let $\mathbb{N}$ denote the set of positive integers and $\mathbb{N}_0=\mathbb{N}\cup\{0\}$. Take $\mathcal{O}^-(x)$ and $\mathcal{O}^+(x)$ to be, respectively, the backward and forward $f-$orbit of the point $x$; that is, $\mathcal{O}^-(x)=\{f^k(x):k\in\mathbb{Z}\setminus\mathbb{N}_0\}$ and $\mathcal{O}^+(x)=\{f^k(x):k\in\mathbb{N}_0\}$.

We denote by $\omega(x)$ or $\omega(x,f)$ the set of all the accumulation points of $\mathcal{O}^+(x)$, which is
$$\omega(x)=\bigcap_{k\geq 1}\overline{\bigcup_{n\geq k}f^n(x)}.$$
We refer to such sets as $\omega$-limit sets. An infinite $\omega$-limit set maximal with respect to the partial order given by set inclusion is said to be of \emph{genus 1} if it does not contain periodic orbits; otherwise, it is of \emph{genus 2}.

The dynamical system $(\mathcal{M},f)$ is said to be minimal if there is not a proper closed subset $X\subsetneq\mathcal{M}$ such that $f(X)=X$. Going back to Birkhoff, there has been a considerable amount of attention dedicated to minimal systems. Among minimal systems, of particular interest, are adding machines. Topologically, they are Cantor sets.

Let us see how they are defined. Let $\alpha=(j_1,j_2,\dots,j_n,\dots)$ be a sequence of integers greater than or equal to 2. Let $\Delta_\alpha$ be the set of all the sequences $(x_1,x_2,\dots,x_n,\dots)$, where $x_n\in\{0,1,\dots,j_n-1\}$ for every $n\in\mathbb{N}$. We can define a metric $d_\alpha$ on $\Delta_\alpha$ so that
$$d_\alpha((x_1,x_2,\dots,x_n,\dots),(y_1,y_2,\dots,y_n,\dots))=\frac{1}{2^n},$$
whenever $(x_1,x_2,\dots,x_n,\dots)\neq(y_1,y_2,\dots,y_n,\dots)$, and $n$ is the first natural number for which $x_n\neq y_n$, and $d_\alpha((x_1,x_2,\dots,x_n,\dots),(y_1,y_2,\dots,y_n,\dots))=0$, otherwise. Given two elements $(x_1,x_2,\dots,x_n,\dots),(y_1,y_2,\dots,y_n,\dots)\in\Delta_\alpha$, we can define
$$(x_1,x_2,\dots,x_n,\dots)+(y_1,y_2,\dots,y_n,\dots)=(z_1,z_2,\dots,z_n,\dots),$$
where, setting $t_0=0$ and $z_1=x_1+y_1$ mod $j_1$, for $n>1$, we take $z_n=x_n+y_n+t_{n-1}$ mod $j_n$ and $t_n=\lfloor\frac{x_n+y_n+t_{n-1}}{j_n}\rfloor$. Here, $\lfloor x\rfloor$ denotes the greatest integer smaller than or equal to $x$.
We define $f_\alpha:\Delta_\alpha\rightarrow\Delta_\alpha$ by
$$f_\alpha((x_1,x_2,\dots,x_n,\dots))=(x_1,x_2,\dots,x_n,\dots)+(1,0,0,\dots,0,\dots).$$
We will refer to the dynamical system $(\Delta_\alpha,f_\alpha)$ as an $\alpha$-\emph{adic adding machine}.

In \cite{BLOCK2004151}, Block and Keesling showed that any infinite minimal dynamical system in which every point is regularly recurrent is conjugate to an adding machine.

\begin{thm}{[1, Theorem 2.3]}
	Let $\alpha\in(\mathbb{N}\setminus\{1\})^\mathbb{N}$. Let $m_i=\prod_{k=1}^{i}j_i$ for every $i\in\mathbb{N}$. Let $f:\mathcal{M}\rightarrow\mathcal{M}$ be a continuous function on a compact metric space $\mathcal{M}$. Then, $f$ is topologically conjugate to $f_\alpha$ if and only if the following three properties hold:
	\begin{enumerate}
		\item for every $i\in\mathbb{N}$, there is a cover $\mathcal{A}_i$ of $\mathcal{M}$ consisting in $m_i$ pairwise disjoint, non-empty and clopen sets which are cyclically permuted by $f$;
		\item for every $i\in\mathbb{N}$, $\mathcal{A}_{i+1}$ is a partition of $\mathcal{A}_i$, and,
		\item if $V_1\supseteq V_2\supseteq\dotsc\supseteq V_n\supseteq\dots$ is a nested sequence with $V_n\in\mathcal{A}_n$ for every $n\in\mathbb{N}$, then $\bigcap_{n=1}^\infty V_n$ is a single point.
	\end{enumerate}
\end{thm}
They were able also to classify adding machines up to topological conjugacy.
\begin{thm}{[1, Corollary 2.8]}
	Let $\alpha=(a_1,a_2,\dots)$ and $\beta=(b_1,b_2,\dots)$ be sequences of integers greater than or equal to 2. Let $M_\alpha$ be a function from the set of prime numbers to the extended natural number defined by
	$$M_\alpha(p)=\sum_{i=1}^\infty n_i,$$
	where $n_i$ is the power of the prime $p$ in the prime factorization of $a_i$.
	
	Then $f_\alpha$ and $f_\beta$ are topologically conjugate if and only if the functions $M_\alpha$ and $M_\beta$ are equal.
\end{thm}
We say that an adding machine is an $\infty$-adic adding machine if $M_\alpha(p)=\infty$ for every prime number $p$. 

The reason why adding machines are well studied is because they arise naturally and frequently in dynamical systems. See for example \cite{MR2475110}. Let $\mathcal{M}$ be a Cantor space or an $n$-manifold and take $C(\mathcal{M},\mathcal{M})$ to be the set of continuous self-maps of $\mathcal{M}$. D'Aniello, Darji and Steele showed there is a residual set $S(\mathcal{M})\subseteq\mathcal{M}\times C(\mathcal{M},\mathcal{M})$ such that $(\omega(x,f),f)$ is an adding machine for every $(x,f)\in S(\mathcal{M})$. They also found that, if $\mathcal{M}$ has the fixed point property, the typical $f\in C(\mathcal{M},\mathcal{M})$ generates uncountably many distinct $\alpha$-adic adding machines for every possible $\alpha$. Moreover, in \cite{MR2593709}, D'Aniello, Humke and Steele proved that the set of $(x,f)\in S(\mathcal{M})$ such that $\omega(x,f)$ is an $\alpha$-adic adding machine is a nowhere dense set with no isolated points, if $\alpha\neq\infty$.

A generalization of the construction for an adding machine leads us to the notion of a simple system and a solenoidal set, and these systems are important because maximal $\omega$-limit sets of genus 1 are example of solenoidal sets. We introduce the concept of a \emph{simple system}, as it is found in \cite{MR1182256}. Let $f:[0,1]\rightarrow[0,1]$ be a continuous $2^\infty$-function with an infinite $\omega$-limit set $\Omega$. Then, there exists a sequence of closed intervals $\{T_k\}_{k=1}^\infty$ such that
\begin{itemize}
	\item for every $k\in\mathbb{N}$, $\{f^i(T_k)\}_{i=1}^{2^k}$ are pairwise disjoint, and $T_k=f^{2^k}(T_k)$;
	\item for every $k\in\mathbb{N}$, $T_{k+1}\cup f^{2^k}(T_{k+1})\subseteq T_k$;
	\item for every $k\in\mathbb{N}$, $\Omega\subseteq\bigcup_{i=1}^{2^k}f^i(T_k)$, and,
	\item for every $i,k\in\mathbb{N}$, $\Omega\cap f^i(T_k)\neq\emptyset$.
\end{itemize}
The set of all the $f^i(T_k)$ forms the \emph{simple system} for $\Omega$ relative to $f$.
Setting
$$K=\bigcap_{k=1}^\infty\bigcup_{i=1}^{2^k}f^i(T_k),$$
let $S$ be the set of components of $K$ with empty interior, and let $Q=\overline{S}$. In \cite{MR1182256} it is shown that
\begin{enumerate}
	\item $(K,f)$ is semi-conjugate to a $(2,2,\dots)$-adic adding machine;
	\item $f\vert Q$ is minimal;
	\item if $[a,b]$ is a connected component of $K$, then $\emptyset\neq[a,b]\cap Q\subseteq\{a,b\}$, and,
	\item there exists a countable $C\subseteq K$ such that $\Omega=Q\dot{\cup}C$, where $\dot{\cup}$ denotes a disjoint union. Moreover, $C$ consist of isolated points in $Q\dot{\cup}C$.
\end{enumerate}
This construction can be generalize to obtain \emph{solenoidal sets} (see \cite{blokh1995spectral}). Let $f:[0,1]\rightarrow[0,1]$ be continuous and let $\Omega$ be a maximal and infinite $\omega$-limit set containing no periodic points. 
Then, there exists a sequence of closed intervals $\{T_k\}_{k=1}^\infty$ and a sequence of natural numbers $(m_1,m_2,\dots)$, where $m_{i+1}$ is a multiple of $m_i$, such that
\begin{itemize}
	\item for every $k\in\mathbb{N}$, $\{f^i(T_k)\}_{i=1}^{m_k}$ are pairwise disjoint, and $T_k=f^{m_k}(T_k)$;
	\item for every $k\in\mathbb{N}$, $T_{k+1}\cup f^{m_k}(T_{k+1})\cup\dots\cup f^{(\frac{m_{k+1}}{m_k}-1)m_k}(T_{k+1})\subseteq T_k$;
	\item for every $k\in\mathbb{N}$, $\Omega\subseteq\bigcup_{i=1}^{m_k}f^i(T_k)$, and,
	\item for every $i,k\in\mathbb{N}$, $\Omega\cap f^i(T_k)\neq\emptyset$.
\end{itemize}
Let
$$K=\bigcap_{k=1}^\infty\bigcup_{i=1}^{m_k}f^i(T_k).$$
Any closed and invariant subset of $K$ is said to be a solenoidal set. 
Let $S$ be the set of components of $K$ with empty interior, and let $Q=\overline{S}$. An analogous result to the previous one can be proved:
\begin{enumerate}
	\item $(K,f)$ is semi-conjugate to a $(m_1,\frac{m_2}{m_1},\frac{m_3}{m_2}\dots)$-adic adding machine;
	\item $f\vert Q$ is minimal;
	\item if $[a,b]$ is a connected component of $K$, then $\emptyset\neq[a,b]\cap Q\subseteq\{a,b\}$, and,
	\item there exists a countable $C\subseteq K$ such that $\Omega=Q\dot{\cup}C$. Moreover, $C$ consist of isolated points in $Q\dot{\cup}C$.
\end{enumerate}

In \cite{dfr}, Della Corte, Farotti and Rodríguez studied the property of complete indecisiveness for competing holes in open dynamical systems. They considered a collection of $N$ shrinking balls $\{\{B^i_n\}_{i=1}^N\}_{n=1}^\infty$ centered at $\{p_i\}_{i=1}^N$ and with radii $\{\{\rho^i_n\}_{n=1}^\infty\}_{i=1}^N$. The authors found out that, under minor hypothesises over the set of centers $\{p_i\}_{i=1}^N$, the typical trajectory changes the escaping hole infinitely many times.

In this paper, we provide some examples in which this happens for the generic element $\{p_i\}_{i=1}^N\in\mathcal{M}^N$. We will refer to such dynamical system as \emph{strachite}. In fact, we will prove that adding machines and the subsystems $(\Omega,f)$ mentioned above are all strachite. Namely, we will see that
\begin{itemize}
	\item adding machines satisfy the hypothesises in \cite{dfr}, but we will see that the set of the possible centers is larger than the one found in that paper;
	\item even though $(\Omega,f)$ is neither transitive nor a homeomorphism in the general case of a simple or solenoidal system, we are able to show $(\Omega,f)$ is strachite, and,
	\item for an interval function, it is enough to have a transitive function in order to have the strachite property.
\end{itemize}

\section{Notations and definitions}

In the following, we will consider $N$-tuples $(p_1,\dots,p_N)\in\mathcal{M}^N$ of pairwise distinct points.
For every $i=1,\dots,N$, let $\{\rho^i_n\}_{n=1}^\infty$ be a decreasing sequence of positive real numbers such that$$\lim_{n\longrightarrow\infty}\rho^i_n=0\text{.}$$

The open ball with center $p_i$ and radius $\rho^i_n$ will be denoted by $B^i_n$; that is, $B^i_n=B_{\rho^i_n}(p_i)$. We refer to the sets $B^i_n$ as holes.

Re-introducing the notation used in \cite{dfr}, by \emph{open dynamical system} we mean a 4-tuple $(\mathcal{M},d,f,\mathfrak{S})$, where $\mathfrak{S}:=\{p_i, \{\rho^i_n\},  i\in \mathbb{N}, n\in\mathbb{N}\}$.

\begin{defi}\label{defT11}
	
	In the following, we adopt the convention $\min \emptyset = +\infty$.
	
	\begin{itemize}
		
		\item A point $x\in \mathcal{M}$ belongs to $\mathfrak{T}(i)$ if 
		the following inequality holds for infinitely many natural numbers $n$:
		$$\min\{k\in\mathbb{N}_0:f^k(x)\in B^{i}_n\}<\min\{k\in\mathbb{N}_0:f^k(x)\in B^j_n\  \text{ for all } j\ne i\}.$$ 
		 
		 \item We say that $x$ belongs to $\mathfrak{T}$ if
		 \begin{equation*}\label{def_compl_indec}
		 	x\in \bigcap_{j=1}^\infty \mathfrak{T}(j).
		 \end{equation*}
	\end{itemize}
	
\end{defi}

\begin{rem}
	For any $i$, a point $x$ in $\mathfrak{T}$ will first visit the hole $B^i_n$, for infinitely many natural numbers $n$.
\end{rem}

Given that we will be working within complete metric spaces, we will be able to make good use of the Baire category theorem.

\begin{defi}
	\noindent A subset $T$ of a topological space $X$ is said to be of \emph{ the first category} if there exists a countable family $\{T_i\}_{i\in\mathbb{N}}$ of nowhere dense subsets of $X$ such that $T=\bigcup_{i\in\mathbb{N}} T_i$. We say that a set $R\subseteq X$ is \emph{residual} if $X\setminus R$ is of the first category. An element of a residual subset of $X$ is called either a \emph{typical} or a \emph{generic} element of $X$.
\end{defi}

\begin{B}\label{th_baire}
	If $X$ is a complete metric space, and $T$ is a first category subset of $X$, then $X\setminus T$ is dense.
\end{B}
In \cite{dfr} the authors define an open dynamical system $(\mathcal{M},d,f,\mathfrak{S}:=\{p_i, \{\rho^i_n\},  i=1,\dots,N,~n\in\mathbb{N}\})$ as \emph{completely indecisive} if the set $\mathfrak{T}$ is residual.

\begin{defi}
	We say a dynamical system $(\mathcal{M},f)$ is \emph{strachite} if, for every natural number $N\in\mathbb{N}$, there exist a residual set $\mathcal{R}\subseteq \mathcal{M}^N$ such that the open dynamical system $(\mathcal{M},d,f,\{p_i, \{\rho^i_n\},  i=1,\dots,N,~n\in\mathbb{N}\})$ is completely indecisive for every $\{p_i\}_{i=1}^N\in\mathcal{R}$ and for every collection of positive decreasing and infinitesimal sequences $\{\{\rho^i_n\}_{n=1}^\infty\}_{i=1}^N$.
\end{defi}

\section{Results}

\begin{thm}
	Let $f_\alpha:\Delta_\alpha\rightarrow\Delta_\alpha$ be an $\alpha$-adic adding machine for some $\alpha\in(\mathbb{N}\setminus\{1\})^\mathbb{N}$. Then, $(\Delta_\alpha,f_\alpha)$ is strachite.
\end{thm}

\begin{proof}
	By \cite{dfr}, it is enough to show that $f_\alpha$ is a transitive homeomorphism. The fact that it is transitive easily follows from the fact that adding machines are minimal systems.
	From the transitivity, it also follows that $f_\alpha$ is surjective. Then, to see that it is a homeomorphism, it suffices to notice that $f_\alpha$ is an isometry.
\end{proof}

\begin{rem}
	In this particular case, since $(\Delta_\alpha,f_\alpha^{-1})$ is minimal, too, the set of the points with a dense backward orbit is not just residual, but it also coincides with the whole space $\Delta_\alpha$. Therefore, in this particular case, the result is slightly stronger than the one described in \cite{dfr}. In general, with a transitive homeomorphism, the set of points with a dense backward orbit is just residual.
\end{rem}

Let us now consider interval maps. Let $f:[0,1]\rightarrow[0,1]$ be a continuous function with a maximal $\omega$-limit set $\omega(x)$ of genus 1. 
We will see that $f$ restricted to $\omega(x)$ is strachite. We proceed by proving some lemmas that will be used to prove the main theorem. Let $K$, $S$, $Q$ and $C$ be as defined in Section \ref{intro} .

\begin{lem}
	$Q\setminus S$ is countable. So, $S$ is residual both in $Q$ and in $Q\dot{\cup}C$.
\end{lem}

\begin{proof}
	If $y\in Q\setminus S$, then there exists an open interval $J\subseteq$Int$(K)$ such that $y\in\partial J$. Thus, it is enough to show that Int$(K)$ consists of at most countably many disjoint intervals. Towards a contradiction, assume Int$(K)=\bigcup_{i\in\mathfrak{I}}J_i$, where the $J_i$ are pairwise disjoint open intervals and $\mathfrak{I}$ is an uncountable set of indices. Then, there exists a natural number $n\in\mathbb{N}$ such that $$\frac{1}{n+1}<\lambda(J_i)\leq\frac{1}{n}$$
	for infinitely many $i$, where $\lambda(J_i)$ denotes the length of the interval $J_i$. But, this would imply Int$(K)\nsubseteq[0,1]$.
\end{proof}

\begin{lem}
	For every $s\in S$, $\overline{\mathcal{O}^-(s)}=Q$.
\end{lem}
\begin{proof}
	Let $s\in S$. Given $q\in Q$, there exists a sequence $s_n\in S$ such that $s_n\xrightarrow{n\rightarrow\infty}q$. We also know that, for every $n\in\mathbb{N}$ and for every $j\in\mathbb{N}$, there exists a natural number $k\in\mathbb{N}$ such that $s_n\in f^k(I_j)$. Therefore, it is enough to show that, for every $j,k\in\mathbb{N}$, $f^k(I_j)\cap \mathcal{O}^-(s)\neq\emptyset$, because this implies that $s_n\in\overline{\mathcal{O}^-(s)}$, for every $n\in\mathbb{N}$.
	
	For every $j\in\mathbb{N}$, there exists $h\in\mathbb{N}$ such that $s\in f^h(I_j)$. Since $I_j$ is periodic, for every $k\in\mathbb{N}$, there exists $l\in\mathbb{N}$ such that $f^l(f^k(I_j))=f^h(I_j)$. So, for every $j,k\in\mathbb{N}$, there exists $x\in f^k(I_j)$ such that $f^l(x)=s$.
\end{proof}

\begin{lem}\label{disj-orb}
	For any $N\in\mathbb{N}$, there exists a residual set $X\subseteq\mathcal{M}^N$ such that, for every $(p_1,\dots,p_N)\in X$, $\mathcal{O}^-(p_i)\cap\mathcal{O}^-(p_j)=\emptyset$, whenever $i\neq j$.
\end{lem}
\begin{proof}
	For every $i,j\in\{1,\dots,N\}$ such that $i\neq j$, and for every $k\in\mathbb{N}$, consider the set
	$$A^{i,j}_k=\{(p_1,\dots,p_N)\in\mathcal{M}^N:f^k(p_j)=p_i\}.$$
	Since $f^k$ is continuous, this set is closed.
	Its complement is dense. In fact, given $(p_1,\dots,p_N)\in A^{i,j}_k$ and $\epsilon>0$, consider $(p_1,\dots,p_i',\dots,p_j,\dots,p_N)$ with $p_i'\neq f^k(p_j)$ and $d(p_i,p_i')<\epsilon$.
	Therefore, the set
	$$X=\mathcal{M}^N\setminus\bigg(\bigcup_{\substack{k\in\mathbb{N}\\i,j\in\{1,\dots,N\}\\i\neq j}}A^{i,j}_k\bigg)$$
	is residual and verifies the statement. In fact, suppose $(p_1,\dots,p_N)\in X$ but $y\in\mathcal{O}^-(p_i)\cap\mathcal{O}^-(p_j)\neq\emptyset$. Then, there exists $n,m\in\mathbb{N}$ such that $f^n(y)=p_i$ and $f^m(y)=p_j$. Suppose $n<m$; then, $f^{m-n}(p_i)=p_j$.
\end{proof}

\begin{thm}\label{w-limit}
	Let $f:[0,1]\rightarrow[0,1]$ be a continuous function. Let $\omega(x)$ be a maximal and infinite $\omega$-limit set of genus 1. Then, $f$ restricted to $\omega(x)$ is strachite.
\end{thm}
\begin{proof}
	Recall that $\omega(x)$ can be written as $Q\dot{\cup}C$. The set $S^N\subseteq\mathcal{M}^N$ is residual, so that $\mathcal{R}=S^N\cap X$ is residual, too, where $X$ is the residual set found in Lemma \ref{disj-orb}. We now show that for every $(p_1,\dots,p_N)\in\mathcal{R}$, the open dynamical system $(\mathcal{M},d,f,\{p_i, \{\rho^i_n\},  i=1,\dots,N,~n\in\mathbb{N}\})$ is completely indecisive.
	
	Given $(p_1,\dots,p_N)\in\mathcal{R}$, it is enough to show $\mathfrak{T}(i)$ is residual, for every $i=1,\dots,N$. As shown in \cite{dfr},
	$$\mathfrak{T}(i)\supseteq\bigcap_{k\geq 0}\bigcup_{n\geq k}\bigcup_{m\geq0}\bigg(f^{-m}(B^i_n)\setminus\bigcup_{h=0}^m\bigcup_{j\neq i}f^{-h}(\overline{B^j_n})\bigg).$$
	Naming
	$$C_n=\bigcup_{m\geq0}\bigg(f^{-m}(B^i_n)\setminus\bigcup_{h=0}^m\bigcup_{j\neq i}f^{-h}(\overline{B^j_n})\bigg),$$
	it is clear that $\bigcup_{n\geq k}C_n$ is open. We now show it is also dense in $Q$, by verifying that it contains $\mathcal{O}^-(p_i)$. Since $\mathcal{O}^-(p_i)\cap\mathcal{O}^-(p_j)=\emptyset$ for every $j\neq i$, the closed sets $f^{-m}(p_i)$ and $\{p_j\}_{j\neq i}\cup\bigcup_{j\neq i}\bigcup_{h=0}^mf^{-h}(p_j)$ are disjoint. Therefore, for a big enough $n\in\mathbb{N}$, $f^{-m}(p_i)\cap\bigcup_{j\neq i}\bigcup_{h=0}^mf^{-h}(\overline{B^j_n})=\emptyset$.
	
	Thus, for every given $(p_1,\dots,p_N)\in\mathcal{R}$, the set of indecisive points of $Q$ is residual in $Q$, and so, it is residual in $Q\dot{\cup}C$, as $C$ is countable.
\end{proof}
\begin{rem}
	Therefore, even though the system $([0,1],f)$ might not be strachite, or even transitive. Just the presence of an infinite $\omega$-limit set with no cycles assures the existence of a subsystem that verifies the strachite property.
\end{rem}
Assuming more, namely that $f:[0,1]\rightarrow[0,1]$ is transitive and continuous, we can see that the whole system has the strachite property.
\begin{lem}
	Let $f:[0,1]\rightarrow[0,1]$ be continuous and transitive. Then $([0,1],f)$ is strachite.
\end{lem}
\begin{proof}
	It is enough to show that the set of points with dense backward orbits is residual. Then, the proof follows the proof of Lemma \ref{disj-orb} and the proof of Theorem \ref{w-limit}. First of all, for every open interval $I\subseteq [0,1]$, we have that $f(I)\subseteq \overline{\text{Int}(f(I))}$. This follows from the fact that $f(I)$ cannot be a single point and so it must be an interval. In fact, let $x\in[0,1]$ be a transitive point and let $I\subseteq[0,1]$ be open such that $f(I)=\{p\}$. Then $x$ visits $I$ twice, and so, there exist $m,n\in\mathbb{N}$ such that $f^m(x)=f^n(x)=p$. So $x$ is periodic, which is a contradiction.
	
	Let $\{B_n\}_{n\in\mathbb{N}}$ be a countable basis for the topology consisting of connected open sets. For the same reason as before, for every $n,k\in\mathbb{N}$, $f^k(B_n)\subseteq\overline{\text{Int}(f^k(B_n))}$.
	 For every $n\in\mathbb{N}$, set
	$$V_n=\bigcup_{k\geq0}f^k(B_n)$$and
	$$V=\bigcap_{n\in\mathbb{N}}V_n.$$
	Since $f$ is transitive, for every opens sets $U,V\subseteq[0,1]$, there exists a natural number $k\in\mathbb{N}$ such that $U\cap f^k(V)\neq\emptyset$ (see \cite{kurka2003topological}); therefore, $V_n$ is dense in $[0,1]$. Since for every $n,k\in\mathbb{N}$, $f^k(B_n)\subseteq\overline{\text{Int}(f^k(B_n))}$, for every $n\in\mathbb{N}$, $V_n$ contains an open subset which is dense in $V_n$. Thus, $V$ is residual.
	Furthermore, $x\in V$ if and only, for every $n\in\mathbb{N}$, there exists $k\in\mathbb{N}$ such that $x\in f^k(B_n)$, which means $\overline{\mathcal{O}^-(x)}=[0,1]$.
\end{proof}

\bibliographystyle{plain}
\bibliography{Adding_machine_and_open_dynamical_systems}

\end{document}